\documentclass[12pt,a4paper]{article}
\usepackage[T1]{fontenc}
\usepackage{amssymb,tikz,graphicx}
\usepackage{lmodern}
\usepackage{amsmath,amssymb,amsthm}
\usepackage[margin=1in]{geometry}
\usepackage{microtype}
\usepackage[labelsep=period]{caption}
\usepackage{placeins}
\usepackage{needspace}
\usepackage[hidelinks]{hyperref}
\newtheorem{theorem}{Theorem}[section]
\newtheorem{corollary}[theorem]{Corollary}
\newtheorem{proposition}[theorem]{Proposition}
\newtheorem{lemma}[theorem]{Lemma}
\newcommand{\Z}{\mathbb Z}
\newcommand{\Pfin}{\mathcal P_{\mathrm{fin}}}
\newcommand{\supp}{\operatorname{supp}}

\title{Thompson's group $F$ is not weakly amenable }
\author{Koji Fujiwara
\thanks{
supported by JSPS Grant-in-Aid for Scientific Research 25H00588.}
}
\date{}
\begin{document}
\maketitle
\vspace{-4mm}

\begin{abstract}
Let $F$ be Thompson's group, and let $D$ be the set of dyadic
rationals in $[0,1]$.
We give an embedding of
$\mathbb Z^{(D)}\rtimes F$ into $F$, where the semidirect product
is defined by the natural action of $F$ on $D$.
Using this embedding,
we deduce from the nonamenability of $F$ that $F$ is not weakly amenable.
\end{abstract}

\section{Introduction}
\subsection{Thompson's group \texorpdfstring{$F$}{F}}
Let $F$ be Thompson's group, the group of  piecewise linear homeomorphisms of
$[0,1]$ with finitely many dyadic breakpoints and slopes in $2^{\Z}$.
A rational number is dyadic if it belongs to
$\Z[1/2]=\{k/2^n:k\in\Z,\ n\geq0\}$. Put
\[
 D=\Z[1/2]\cap[0,1],\qquad D_0=D\setminus\{0,1\}.
\]
The group $F$ acts  on $D$ and preserves $D_0$.

A \emph{mean} on a set $Y$ is a finitely
additive probability measure on all subsets of $Y$. It satisfies
$m(Y)=1$, $m(A)\geq0$, and $m(A\cup B)=m(A)+m(B)$ for disjoint
$A,B\subseteq Y$. For an action of $G$ on $Y$, it is
\emph{invariant} if $m(gA)=m(A)$ for all $g\in G$ and $A\subseteq Y$.
A group is \emph{amenable} if its action on itself by left translations
admits an invariant mean.

OpenAI \cite{OpenAI} proved that $F$ is nonamenable.
This has been a long standing question. 
\begin{theorem}[OpenAI]\label{thm:noname}
    Thompson's group $F$ is not  amenable.
\end{theorem}

There is a weaker notion of amenability. 
A countable discrete group $G$ is \emph{weakly amenable},
in the sense of Cowling and Haagerup, if there is a sequence
of finitely supported functions $\varphi_n:G\to\mathbb C$
converging pointwise to $1$ such that
there are Hilbert spaces $H_n$, maps
$\xi_n,\eta_n:G\to H_n$, and a constant $C<\infty$ with
\[
\varphi_n(y^{-1}x)=\langle\xi_n(x),\eta_n(y)\rangle,
\qquad
\|\xi_n(x)\|\le C,
\qquad
\|\eta_n(y)\|\le C
\]
for every $n$ and $x,y\in G$.
See \cite[Section~5, p.~1392]{GRT}.

Every subgroup of a weakly amenable group is weakly amenable.
This follows immediately by restricting $\varphi_n$, $\xi_n$,
and $\eta_n$ to the subgroup. As the name suggests, every amenable countable discrete group is weakly amenable
\cite[Section~1]{Oz}.
Also, every hyperbolic group is weakly amenable
\cite[Theorem~2]{OzHyp}.

We give a negative answer to the question of whether $F$ is weakly
amenable, mentioned by Guentner and Higson \cite[Introduction]{GH}. Our argument does not give a new proof
of nonamenability of $F$. We use the nonamenability of $F$.

\begin{theorem}\label{thm:nonwa}
Thompson's group $F$ is not weakly amenable.
\end{theorem}

We recall the definition of extensive amenability.
Let $G$ be a group acting on a set $X$.
Write $\Pfin(X)=\{P\subseteq X:P\text{ is finite}\}$, with the
induced action $gP=\{gx:x\in P\}$. An action of $G$ on $X$
is \emph{extensively amenable} if there is a $G$-invariant mean $\nu$
on $\Pfin(X)$ such that for every finite $M\subseteq X$, 
\begin{equation*}\label{eq:ea}
 \nu\{P\in\Pfin(X):M\subseteq P\}=1.
\end{equation*}
Every action of an amenable group is extensively amenable
\cite[Lemma~2.1]{JMMdS}.
See \cite{JMMdS} for more information on this notion.

Kim and Lodha \cite{KL} gave a streamlined proof of Theorem \ref{thm:noname}.
Their proof explicitly establishes the following statement \cite[Section~2]{KL}.

\begin{theorem}[Kim--Lodha, {\cite[Section~2]{KL}}]\label{thm:kl}
The natural action of $F$ on $D$ is not extensively amenable.
\end{theorem}

 Since every action of an amenable group is extensively amenable,
Theorem~\ref{thm:kl} implies Theorem \ref{thm:noname}.
Chornyi \cite[Corollary~3]{Chornyi} proved that $F$ is amenable
if and only if the action of $F$ on $D_0$ is extensively
amenable. Since $F$ fixes $0$ and $1$, extensive amenability of the actions
on $D_0$ and on $D$ is equivalent. Thus Theorem~\ref{thm:kl} also follows from OpenAI's theorem
and Chornyi's equivalence.

\subsection{Weak amenability and extensive amenability}

Let $X$ be a set.
For a function $\eta:X\to\Z$, define its support by
$
\supp(\eta)=\{x\in X:\eta(x)\ne0\}.
$
We say that $\eta$ is finitely supported if $\supp(\eta)$ is finite.
Write $\Z^{(X)}$ for the set of all finitely supported functions
from $X$ to $\Z$.
Addition is defined by
$
(\eta+\theta)(x)=\eta(x)+\theta(x).
$
With this addition, $\Z^{(X)}$ is an abelian group.

Let $G$ be a group acting on $X$.
For $g\in G$ and $\eta\in\Z^{(X)}$, define $g\cdot\eta$ by
\[
(g\cdot\eta)(x)=\eta(g^{-1}x)
\qquad (x\in X).
\]
This defines an action of $G$ on $\Z^{(X)}$ by group automorphisms.
The semidirect product $\Z^{(X)}\rtimes G$ consists of pairs
$(\eta,g)$ with $\eta\in\Z^{(X)}$ and $g\in G$.
Its multiplication is
\[
(\eta,g)(\theta,h)=(\eta+g\cdot\theta,gh).
\]
This group is called the {\it permutational wreath product} of $\Z$
by $G$ for the given action on $X$.

The following proposition relates weak amenability
and extensive amenability.
\begin{proposition}\label{lem:mean}
Let $G$ be a countable discrete group acting on a countable  set $X$.
If $\Z^{(X)}\rtimes G$ is weakly amenable, then
the action of $G$ on $X$ is extensively amenable.
\end{proposition}

Proposition \ref{lem:mean} follows from 
Ozawa's theorem and \cite[Theorem 1.3]{JMMdS}.
For the reader's convenience, we give a proof in Section \ref{sec:mean}.

\subsection{Embedding the wreath product in \texorpdfstring{$F$}{F}}

We give an explicit and elementary construction of the following embedding.

\begin{proposition}\label{prop:embedding}
For the natural action of $F$ on $D$, there is an injective
homomorphism
\[
 \Z^{(D)}\rtimes F\longrightarrow F.
\]
\end{proposition}

We prove Proposition \ref{prop:embedding} in Section \ref{sec:embed}.
The proof uses neither Theorem~\ref{thm:kl} nor Theorem \ref{thm:nonwa}.
Combining Propositions \ref{lem:mean} and \ref{prop:embedding},
we obtain the following. 

\begin{corollary}\label{prop:reduction}
If $F$ is weakly amenable, then the action of $F$ on $D$ is extensively
amenable.
\end{corollary}

\begin{proof}
Suppose that $F$ is weakly amenable. By
Proposition~\ref{prop:embedding},
$\Z^{(D)}\rtimes F$ is isomorphic to a subgroup of $F$.
Since weak amenability passes to subgroups,
$\mathbb Z^{(D)}\rtimes F$ is weakly amenable.
By Proposition~\ref{lem:mean}, the action of 
$F$ on $D$ is extensively amenable. 
\end{proof}

\begin{corollary}\label{cor:waa}
    If $F$ is weakly amenable, then $F$ is amenable.
\end{corollary}

\begin{proof}
By Corollary~\ref{prop:reduction}, the action of $F$ on $D$ is extensively
amenable. Chornyi's equivalence \cite[Corollary~3]{Chornyi} then implies that $F$ is amenable.
\end{proof}

\begin{proof}[Proof of Theorem~\ref{thm:nonwa}]
If $F$ were weakly amenable, Corollary~\ref{cor:waa}
would imply that $F$ is amenable, 
contradicting Theorem~\ref{thm:noname}.
\end{proof}
Alternatively, Corollary~\ref{prop:reduction} shows that weak amenability of $F$
would contradict Theorem~\ref{thm:kl}.

\section{Proof of Proposition \ref{lem:mean}}\label{sec:mean}

We recall the following theorem. 

\begin{theorem}[Ozawa, {\cite[Theorem~A]{Oz}}]\label{thm:ozawa}
Let $W$ be a countable weakly amenable discrete group and let $N\lhd W$ be
amenable. Then there is a mean $m$ on $N$ invariant under left translations
by $N$ and conjugation by $W$. 
\end{theorem}
Proposition \ref{lem:mean}  follows from
Theorem~\ref{thm:ozawa} and a general result
\cite[Theorem~1.3]{JMMdS}, applied to the functor
$Y\mapsto\mathbb Z^{(Y)}$. We include a direct proof of this special case.

\begin{proof}
By assumption, $\Z^{(X)}\rtimes G$ is weakly amenable.
The subgroup $\Z^{(X)}$ is abelian, hence amenable, and normal
in $\Z^{(X)}\rtimes G$.
Theorem~\ref{thm:ozawa} gives a mean $m$ on $\Z^{(X)}$
invariant under translations by $\Z^{(X)}$ and conjugation by $G$.

For $x\in X$ and $k\in\Z$, put
$
A_{x,k}=\{\eta\in\Z^{(X)}:\eta(x)=k\}.
$
We first show that $m(A_{x,0})=0$ for every $x\in X$.
Define $e_x\in\Z^{(X)}$ by $e_x(x)=1$ and
$e_x(y)=0$ for $y\ne x$.
Translation by $e_x$  maps
$A_{x,k}$ onto $A_{x,k+1}$.
Translation invariance of $m$ gives
$
m(A_{x,k})=m(A_{x,0})$ for every $k \in \Z$.
For every positive integer $r$, the sets
$A_{x,0},\ldots,A_{x,r-1}$ are pairwise disjoint.
Finite additivity gives
$
r\,m(A_{x,0}) \le 1$.
Thus 
$m(A_{x,0})=0$.

Let $M\subseteq X$ be finite.
The set $\bigcup_{x\in M}A_{x,0}$ consists of the functions
that vanish at some point of $M$.
Since  $m(A_{x,0})=0$ for every $x$, 
\[
0\le m\Bigl(\bigcup_{x\in M}A_{x,0}\Bigr)
\le \sum_{x\in M}m(A_{x,0})=0.
\]
The complement of  $\bigcup_{x\in M}A_{x,0}$ consists of the functions whose supports contain $M$.
Therefore
\begin{equation}\label{eq:support}
m\{\eta\in\Z^{(X)}:M\subseteq\supp(\eta)\}
=1-m\Bigl(\bigcup_{x\in M}A_{x,0}\Bigr)=1.
\end{equation}

Now, for each family $\mathcal A\subseteq\Pfin(X)$, define
\[
\nu(\mathcal A)
=m\{\eta\in\Z^{(X)}:\supp(\eta)\in\mathcal A\}.
\]
We have $\nu(\mathcal A)\ge0$ because $m$ is nonnegative.
Every function in $\Z^{(X)}$ has finite support, so
$\nu(\Pfin(X))=m(\Z^{(X)})=1$.
For disjoint families $\mathcal A$ and $\mathcal B$,
the corresponding sets of functions are disjoint.
Finite additivity of $m$ therefore gives
$
\nu(\mathcal A\cup\mathcal B)
=\nu(\mathcal A)+\nu(\mathcal B).
$
Hence $\nu$ is a mean on $\Pfin(X)$.

We identify $\eta\in\Z^{(X)}$ with $(\eta,1)$ and
$g\in G$ with $(0,g)$ in $\Z^{(X)}\rtimes G$.
Since
\[
(0,g)(\eta,1)(0,g)^{-1}
=(g\cdot\eta,g)(0,g^{-1})
=(g\cdot\eta,1),
\]
 conjugation by $(0,g)$ acts on $\Z^{(X)}$
by $\eta\mapsto g\cdot\eta$.

Let $g\in G$ and $\mathcal A\subseteq\Pfin(X)$.
Write $g\mathcal A=\{gP:P\in\mathcal A\}$.
Since
$
\supp(g\cdot\eta)=g\supp(\eta),
$
conjugation by $(0,g)$ maps the set
$
\{\eta\in\Z^{(X)}:\supp(\eta)\in\mathcal A\}
$
onto the set
$
\{\eta\in\Z^{(X)}:\supp(\eta)\in g\mathcal A\}.
$

Since $m$ is invariant under conjugation, it takes the same
value on these two sets.
By the definition of $\nu$, these values are
$\nu(\mathcal A)$ and $\nu(g\mathcal A)$.
Thus $\nu(g\mathcal A)=\nu(\mathcal A)$, proving that
$\nu$ is $G$-invariant.

Finally, for every finite $M\subseteq X$, equation~\eqref{eq:support}
gives
\[
\nu\{P\in\Pfin(X):M\subseteq P\}
=m\{\eta\in\Z^{(X)}:M\subseteq\supp(\eta)\}=1.
\]
Hence the action of  $G$ on $X$ is extensively amenable.
\end{proof}

\section{Proof of Proposition \ref{prop:embedding}}\label{sec:embed}
We replace each point of $D_0$ by an interval and lift the action of $F$.
Our construction is closely related to
\cite[Example~2.14]{BF}, where Belk--Forrest realize
$F^{(D_0)}\rtimes F$, for the natural action on $D_0$,
as a group of homeomorphisms of an interval.
We give a direct dyadic construction below.

We denote the lifted action by $\rho$.
The main step is to prove that $\rho(f)\in F$ for every $f\in F$.
This is Lemma~3.1.
We then put a copy of $\mathbb Z$ in each inserted interval
to obtain an embedding of $\mathbb Z^{(D_0)}\rtimes F$ into $F$.
Finally, we pass from $D_0$ to $D$.

\bigskip
\noindent\emph{Inserting intervals into $[0,1]$.}
To construct a homomorphism $\rho\colon F\to F$,
we first define  intervals $J_d \subset (0,1)$ for every $d \in D_0$.
Use an original copy and a new copy of $[0,1]$.
First bisect the original interval at $1/2$.
In the new copy, put $J_{1/2}=(1/2,3/4)$.
The closed intervals $[0,1/2]$ and $[3/4,1]$ on its two sides
correspond to the original halves $[0,1/2]$ and $[1/2,1]$, respectively.

At each subsequent step, bisect every interval of the original subdivision.
Divide each corresponding interval in the new copy into three pieces,
from left to right, in the proportions $1/2,1/4,1/4$.
If $d \in D_0$ is the midpoint of the original interval, call the interior
of the middle piece $J_d$.
The left and right pieces correspond to the two original halves.
For example, the second step bisects the original halves at $1/4$ and $3/4$.
In the new copy this gives $J_{1/4}=(1/4,3/8)$ and
$J_{3/4}=(7/8,15/16)$.
See Figure~\ref{fig:inserting-intervals}.

Repeat this procedure indefinitely.
Every $d\in D_0$ occurs as an original midpoint after finitely many
bisections, so this defines $J_d$ for every $d\in D_0$.
Once inserted, each $J_d$ is left unchanged.
At each step, the whole interval remains $[0,1]$.
All interval endpoints in this construction are dyadic, and all interval
lengths are powers of $2$.
The intervals $J_d$ occur in the same order as the points $d$.
After \(n\) steps, every remaining closed interval has length at most \(2^{-n}\).
Thus the union of all $J_d$ is dense in $[0,1]$.

\begin{figure}[htbp]
\centering
\begin{tikzpicture}[x=1cm,y=1cm,font=\small,line width=.55pt]
\node[anchor=east] at (-.2,3.1) {Original};
\draw (0,3.1)--(10,3.1);
\foreach \x/\t in {0/0,2.5/{1/4},5/{1/2},7.5/{3/4},10/1}{
\fill (\x,3.1) circle (1.2pt);
\node[above=4pt] at (\x,3.1) {$\t$};}
\node[anchor=east] at (-.2,1.8) {Step 1};
\draw (0,1.6) rectangle (5,2);
\draw[fill=gray!15] (5,1.6) rectangle (7.5,2);
\draw (7.5,1.6) rectangle (10,2);
\node at (6.25,1.8) {$J_{1/2}$};
\foreach \x/\t in {0/0,5/{1/2},7.5/{3/4},10/1}
\node[below=3pt,font=\scriptsize] at (\x,1.6) {$\t$};
\node[anchor=east] at (-.2,.3) {Step 2};
\foreach \a/\b in {0/2.5,3.75/5,7.5/8.75,9.375/10}
\draw (\a,.1) rectangle (\b,.5);
\foreach \a/\b/\t in {2.5/3.75/{J_{1/4}},5/7.5/{J_{1/2}},8.75/9.375/{J_{3/4}}}{
\draw[fill=gray!15] (\a,.1) rectangle (\b,.5);}
\node at (3.125,.3) {$J_{1/4}$};
\node at (6.25,.3) {$J_{1/2}$};
\node[above=3pt] at (9.0625,.5) {$J_{3/4}$};
\foreach \x/\t in {0/0,10/1}
\node[below=3pt,font=\scriptsize] at (\x,.1) {$\t$};
\end{tikzpicture}
\caption{The first two steps. The subscript of $J_d$ is the position of
$d$ in the original interval.}
\label{fig:inserting-intervals}
\end{figure}
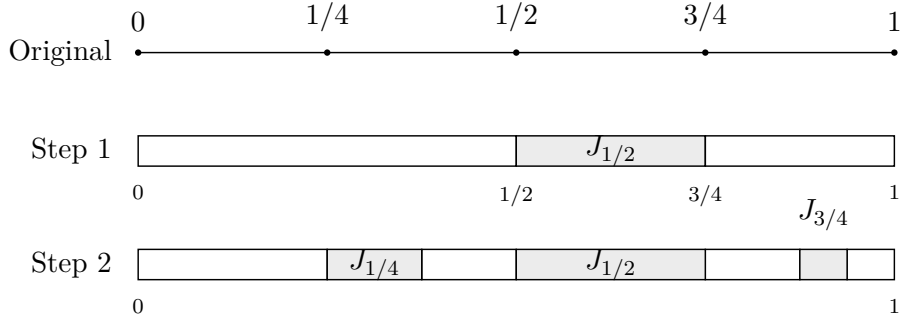

\medskip
\Needspace{6\baselineskip}
\noindent\emph{Defining $\rho(f)$.}
Let $f\in F$.
On each $J_d$, define $\rho(f)$ to be the increasing affine bijection
onto $J_{f(d)}$.
Since $f$ is increasing and permutes $D_0$, these maps together give an
increasing bijection of the dense set $\bigcup_{d\in D_0}J_d$ onto itself.
It extends uniquely to an increasing homeomorphism of $[0,1]$,
which we also denote by $\rho(f)$.
On each $J_d$, its slope is the ratio of the lengths of $J_{f(d)}$ and
$J_d$, and hence is a power of $2$.

\medskip
\noindent\emph{The map $\rho$ is a homomorphism.}
Let $f,g\in F$.
The map $\rho(g)$ sends $J_d$ affinely onto $J_{g(d)}$, and
$\rho(f)$ sends this interval affinely onto $J_{f(g(d))}$.
Thus $\rho(f)\rho(g)$ and $\rho(fg)$ are the same increasing affine
map on each $J_d$.
Since the union of the $J_d$ is dense and both maps are continuous,
they agree on $[0,1]$.
Hence $\rho(fg)=\rho(f)\rho(g)$.

\medskip
Let
\[
\mathcal C=[0,1]\setminus\bigcup_{d\in D_0}J_d.
\]

\Needspace{8\baselineskip}
\noindent\emph{The action on $\mathcal C$ is faithful.}
The maps $\rho(f)$ preserve $\mathcal C$.
If $\rho(f)$ fixes $\mathcal C$ pointwise, it fixes both endpoints of every
$J_d$. Since it maps $J_d$ onto $J_{f(d)}$, we have $f(d)=d$ for every
$d\in D_0$, so $f=\mathrm{id}$.
Thus the action of $F$ on $\mathcal C$ defined by $\rho$ is faithful.

It remains to show that $\rho(f)$ belongs to $F$.

\Needspace{7\baselineskip}
\begin{lemma}\label{lem:rho-in-F}
For every $f\in F$, the map $\rho(f)$ belongs to $F$.
\end{lemma}

\begin{proof}
Fix $f\in F$.
We use finitely many bisections to show that $\rho(f)$ has finitely many
affine pieces.
First divide the domain of $f$ into $2^n$ equal intervals, with $n$
large enough that $f$ is affine on each interval.
Next, bisect the image of each interval as needed until each piece is one of
the intervals in a subdivision of $[0,1]$ into $2,4,8,\ldots$ equal parts.
This is possible because the image endpoints are dyadic and the image
lengths are powers of $2$.
Pull these subdivisions back by $f$.
Since $f$ is affine on each domain interval, this also subdivides the
domain intervals by bisection.
Finally, increase $n$ if necessary and refine the domain to $2^n$ equal
intervals. Their images still have the required form, but need not have
the same length.

Figure~\ref{fig:rho} (left) shows why further bisections of the image of an interval may be needed.
The map $f$ is affine on $[1/4,1/2]$, but its image $[1/4,3/4]$
is not an interval obtained by repeated bisection of $[0,1]$.
Bisecting this image at $1/2$ adds the partition point $3/8$ in the domain.
The figure uses eight equal intervals in the domain.

%All the intervals $J_d$ have already been inserted.
Let $d_1<\cdots<d_m \, (m=2^n-1)$ be the subdivision points in
$(0,1)$ of the original domain chosen above.
Consider $J_{d_1},\ldots,J_{d_m}$ in the domain and
$J_{f(d_1)},\ldots,J_{f(d_m)}$ in the image.
Divide the domain at the endpoints of \(J_{d_1},\ldots,J_{d_m}\), and the image at the endpoints of \(J_{f(d_1)},\ldots,J_{f(d_m)}\). 
See Figure~\ref{fig:rho} (right).

For $1\leq i<m$, let $K_i$ be the closed interval between
$J_{d_i}$ and $J_{d_{i+1}}$.
Let $K_i'$ be the closed interval between
$J_{f(d_i)}$ and $J_{f(d_{i+1})}$.
We show that $\rho(f)$ is affine on $K_i$.
By our choice of $d_i$'s, 
both $[d_i,d_{i+1}]$ and its image under $f$
are obtained by repeated bisection of $[0,1]$.

The interval $K_i$ contains every $J_d$ with $d_i<d<d_{i+1}$.
Inside $K_i$ and $K_i'$, the intervals are inserted using the same
proportions $1/2,1/4,1/4$ at each step.
Since $f$ is affine on $[d_i,d_{i+1}]$, it sends each successive
bisection point of this interval to the corresponding bisection point
of its image.
Thus the increasing affine map from $K_i$ onto $K_i'$ sends every
$J_d\subset K_i$ onto $J_{f(d)}$ and therefore agrees with $\rho(f)$
on each $J_d$.
Their union is dense in $K_i$, so the two maps agree on all of $K_i$
by continuity.
The same argument applies to the two closed intervals next to $0$ and $1$.

On each selected $J_d$, the map $\rho(f)$ is affine by definition.
Thus it is affine on every piece of the finite partition.
Each piece and its image have dyadic endpoints and length a power of $2$.
It follows that $\rho(f)$ has finitely many dyadic breakpoints and
slopes in $2^{\mathbb Z}$.
Hence $\rho(f)\in F$.
\end{proof}

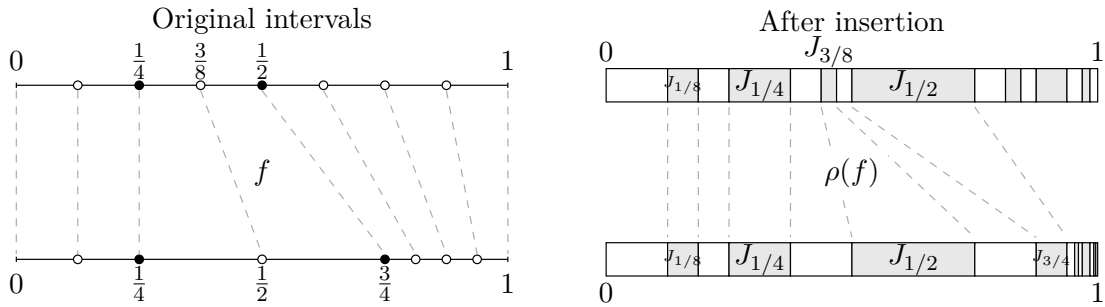
\begin{figure}[htbp]
\centering
\begin{tikzpicture}[x=1cm,y=1cm,font=\small,line width=.45pt]
\node at (3.25,3.15) {Original intervals};
\draw (0,2.3)--(6.5,2.3);
\draw (0,2.265)--(0,2.335);
\draw (0.8125,2.265)--(0.8125,2.335);
\draw (1.625,2.265)--(1.625,2.335);
\draw (2.4375,2.265)--(2.4375,2.335);
\draw (3.25,2.265)--(3.25,2.335);
\draw (4.0625,2.265)--(4.0625,2.335);
\draw (4.875,2.265)--(4.875,2.335);
\draw (5.6875,2.265)--(5.6875,2.335);
\draw (6.5,2.265)--(6.5,2.335);
\node at (0,2.62) {$0$};
\node at (1.625,2.62) {$\frac{1}{4}$};
\node at (2.4375,2.62) {$\frac{3}{8}$};
\node at (3.25,2.62) {$\frac{1}{2}$};
\node at (6.5,2.62) {$1$};
\draw (0,0)--(6.5,0);
\draw (0,-0.035)--(0,0.035);
\draw (0.8125,-0.035)--(0.8125,0.035);
\draw (1.625,-0.035)--(1.625,0.035);
\draw (3.25,-0.035)--(3.25,0.035);
\draw (4.875,-0.035)--(4.875,0.035);
\draw (5.28125,-0.035)--(5.28125,0.035);
\draw (5.6875,-0.035)--(5.6875,0.035);
\draw (6.09375,-0.035)--(6.09375,0.035);
\draw (6.5,-0.035)--(6.5,0.035);
\node at (0,-0.32) {$0$};
\node at (1.625,-0.32) {$\frac{1}{4}$};
\node at (3.25,-0.32) {$\frac{1}{2}$};
\node at (4.875,-0.32) {$\frac{3}{4}$};
\node at (6.5,-0.32) {$1$};
\draw[gray!65,dashed] (0,2.23)--(0,.07);
\draw[gray!65,dashed] (0.8125,2.23)--(0.8125,.07);
\draw[fill=white] (0.8125,2.3) circle (1.6pt);
\draw[fill=white] (0.8125,0) circle (1.6pt);
\draw[gray!65,dashed] (1.625,2.23)--(1.625,.07);
\draw[fill=black] (1.625,2.3) circle (1.6pt);
\draw[fill=black] (1.625,0) circle (1.6pt);
\draw[gray!65,dashed] (2.4375,2.23)--(3.25,.07);
\draw[fill=white] (2.4375,2.3) circle (1.6pt);
\draw[fill=white] (3.25,0) circle (1.6pt);
\draw[gray!65,dashed] (3.25,2.23)--(4.875,.07);
\draw[fill=black] (3.25,2.3) circle (1.6pt);
\draw[fill=black] (4.875,0) circle (1.6pt);
\draw[gray!65,dashed] (4.0625,2.23)--(5.28125,.07);
\draw[fill=white] (4.0625,2.3) circle (1.6pt);
\draw[fill=white] (5.28125,0) circle (1.6pt);
\draw[gray!65,dashed] (4.875,2.23)--(5.6875,.07);
\draw[fill=white] (4.875,2.3) circle (1.6pt);
\draw[fill=white] (5.6875,0) circle (1.6pt);
\draw[gray!65,dashed] (5.6875,2.23)--(6.09375,.07);
\draw[fill=white] (5.6875,2.3) circle (1.6pt);
\draw[fill=white] (6.09375,0) circle (1.6pt);
\draw[gray!65,dashed] (6.5,2.23)--(6.5,.07);
\node[fill=white,inner sep=3pt] at (3.25,1.15) {$f$};
\begin{scope}[xshift=7.8cm]
\node at (3.25,3.15) {After insertion};
\draw[fill=white] (0,2.08) rectangle (0.8125,2.52);
\draw[fill=gray!18] (0.8125,2.08) rectangle (1.21875,2.52);
\draw[fill=white] (1.21875,2.08) rectangle (1.625,2.52);
\draw[fill=gray!18] (1.625,2.08) rectangle (2.4375,2.52);
\draw[fill=white] (2.4375,2.08) rectangle (2.84375,2.52);
\draw[fill=gray!18] (2.84375,2.08) rectangle (3.046875,2.52);
\draw[fill=white] (3.046875,2.08) rectangle (3.25,2.52);
\draw[fill=gray!18] (3.25,2.08) rectangle (4.875,2.52);
\draw[fill=white] (4.875,2.08) rectangle (5.28125,2.52);
\draw[fill=gray!18] (5.28125,2.08) rectangle (5.484375,2.52);
\draw[fill=white] (5.484375,2.08) rectangle (5.6875,2.52);
\draw[fill=gray!18] (5.6875,2.08) rectangle (6.09375,2.52);
\draw[fill=white] (6.09375,2.08) rectangle (6.296875,2.52);
\draw[fill=gray!18] (6.296875,2.08) rectangle (6.3984375,2.52);
\draw[fill=white] (6.3984375,2.08) rectangle (6.5,2.52);
\node at (0,2.74) {$0$};
\node at (6.5,2.74) {$1$};
\node[font=\fontsize{6}{7}\selectfont,inner sep=0pt] at (1.015625,2.3) {$J_{1/8}$};
\node[font=\small,inner sep=0pt] at (2.03125,2.3) {$J_{1/4}$};
\node[font=\small,inner sep=0pt] at (4.0625,2.3) {$J_{1/2}$};
\draw[fill=white] (0,-0.22) rectangle (0.8125,0.22);
\draw[fill=gray!18] (0.8125,-0.22) rectangle (1.21875,0.22);
\draw[fill=white] (1.21875,-0.22) rectangle (1.625,0.22);
\draw[fill=gray!18] (1.625,-0.22) rectangle (2.4375,0.22);
\draw[fill=white] (2.4375,-0.22) rectangle (3.25,0.22);
\draw[fill=gray!18] (3.25,-0.22) rectangle (4.875,0.22);
\draw[fill=white] (4.875,-0.22) rectangle (5.6875,0.22);
\draw[fill=gray!18] (5.6875,-0.22) rectangle (6.09375,0.22);
\draw[fill=white] (6.09375,-0.22) rectangle (6.1953125,0.22);
\draw[fill=gray!18] (6.1953125,-0.22) rectangle (6.24609375,0.22);
\draw[fill=white] (6.24609375,-0.22) rectangle (6.296875,0.22);
\draw[fill=gray!18] (6.296875,-0.22) rectangle (6.3984375,0.22);
\draw[fill=white] (6.3984375,-0.22) rectangle (6.44921875,0.22);
\draw[fill=gray!18] (6.44921875,-0.22) rectangle (6.474609375,0.22);
\draw[fill=white] (6.474609375,-0.22) rectangle (6.5,0.22);
\node at (0,-0.44) {$0$};
\node at (6.5,-0.44) {$1$};
\node[font=\fontsize{6}{7}\selectfont,inner sep=0pt] at (1.015625,0) {$J_{1/8}$};
\node[font=\small,inner sep=0pt] at (2.03125,0) {$J_{1/4}$};
\node[font=\small,inner sep=0pt] at (4.0625,0) {$J_{1/2}$};
\node[font=\fontsize{6}{7}\selectfont,inner sep=0pt] at (5.890625,0) {$J_{3/4}$};
\node at (2.9453125,2.74) {$J_{3/8}$};
\draw[gray!65,dashed] (0.8125,2.01)--(0.8125,.29);
\draw[gray!65,dashed] (1.21875,2.01)--(1.21875,.29);
\draw[gray!65,dashed] (1.625,2.01)--(1.625,.29);
\draw[gray!65,dashed] (2.4375,2.01)--(2.4375,.29);
\draw[gray!65,dashed] (2.84375,2.01)--(3.25,.29);
\draw[gray!65,dashed] (3.046875,2.01)--(4.875,.29);
\draw[gray!65,dashed] (3.25,2.01)--(5.6875,.29);
\draw[gray!65,dashed] (4.875,2.01)--(6.09375,.29);
\node[fill=white,inner sep=3pt] at (3.25,1.15) {$\rho(f)$};
\end{scope}
\end{tikzpicture}
\caption{Finite partitions for $f$ and $\rho(f)$, where
$f(t)=t$ on $[0,1/4]$, $f(t)=2t-1/4$ on $[1/4,1/2]$, and
$f(t)=(t+1)/2$ on $[1/2,1]$.
Both maps act from top to bottom.
Filled circles mark the breakpoints of $f$ and their images.
Open circles mark the other partition points and their images.
The shaded intervals are the selected $J_d$.
The map $\rho(f)$ is affine on each shaded or white interval.
}
\label{fig:rho}
\end{figure}
\FloatBarrier

\begin{proof}[Proof of Proposition~\ref{prop:embedding}]
By Lemma~\ref{lem:rho-in-F}, the homomorphism $\rho$ takes values in $F$.
Its action on $\mathcal C$ is faithful.

\medskip
\noindent\textbf{Embedding $\mathbb Z^{(D_0)}$ into $F$.}
We realize the copy of $\mathbb Z$ indexed by $d$ inside the interval $J_d$.
Fix a nontrivial $z\in F$. It has infinite order.
For each $d\in D_0$, let $a_d\colon[0,1]\to\overline{J_d}$ be the
increasing affine bijection.
We define $z_d$ to equal $a_dza_d^{-1}$ on $\overline{J_d}$ and
the identity elsewhere.

The map $a_d$ has dyadic translation part and slope a power of $2$.
Hence $z_d$ has finitely many dyadic breakpoints and slopes in
$2^{\mathbb Z}$, so $z_d\in F$.
The maps $z_d$ commute because they act in disjoint intervals.

Define
$
L\colon\mathbb Z^{(D_0)}\to F,
$
by
$$
L(\eta)=\prod_{d\in\supp(\eta)}z_d^{\eta(d)}.
$$
The product is finite because $\eta$ has finite support.
On each $J_d$, the map $L(\eta)$ agrees with $z_d^{\eta(d)}$.
It fixes every point of $\mathcal C$.
Since the $z_d$ commute, we have $L(\eta+\theta)=L(\eta)L(\theta)$.
Thus $L$ is a homomorphism.

If $L(\eta)=\mathrm{id}$, then $z_d^{\eta(d)}=\mathrm{id}$ for every $d$.
Since each $z_d$ has infinite order, $\eta(d)=0$ for every $d$.
Thus $L$ is injective.

\medskip
\Needspace{6\baselineskip}
\noindent\textbf{Embedding $\mathbb Z^{(D_0)}\rtimes F$ into $F$.}
Recall that on $\overline{J_d}$, the map $z_d$ is obtained by
conjugating $z$ by the increasing affine map $a_d$.
Applying $a_d$ and then $\rho(f)$ gives the increasing affine map
from $[0,1]$ onto $\overline{J_{f(d)}}$, which is $a_{f(d)}$.
Thus conjugating $z_d$ by $\rho(f)$ is the same as conjugating $z$
by $a_{f(d)}$, giving $z_{f(d)}$.
For each $d\in D_0$ and $f\in F$, we therefore have
$\rho(f)z_d\rho(f)^{-1}=z_{f(d)}$.
Hence
\begin{equation}\label{eq:rho-conjugation}
\begin{aligned}
\rho(f)L(\eta)\rho(f)^{-1}
&=\prod_{d\in\supp(\eta)}
  \bigl(\rho(f)z_d\rho(f)^{-1}\bigr)^{\eta(d)}\\
&=\prod_{d\in\supp(\eta)}z_{f(d)}^{\eta(d)}=L(f\cdot\eta).
\end{aligned}
\end{equation}
The last equality uses $(f\cdot\eta)(f(d))=\eta(d)$.

Define
$
\Phi\colon\mathbb Z^{(D_0)}\rtimes F\to F$
by 
$
\Phi(\eta,f)=L(\eta)\rho(f).
$
Using \eqref{eq:rho-conjugation}, we have
\[
\begin{aligned}
\Phi(\eta,f)\Phi(\theta,g)
&=L(\eta)\rho(f)L(\theta)\rho(g)
 =L(\eta)L(f\cdot\theta)\rho(f)\rho(g)\\
&=L(\eta+f\cdot\theta)\rho(fg)
 =\Phi(\eta+f\cdot\theta,fg).
\end{aligned}
\]
Hence $\Phi$ is a homomorphism.

To prove that $\Phi$ is injective, suppose that $\Phi(\eta,f)=\mathrm{id}$.
Then $\rho(f)=L(\eta)^{-1}$.
Since $L(\eta)$ fixes every point of $\mathcal C$, so does $\rho(f)$.
The action of $F$ on $\mathcal C$ defined by $\rho$ is faithful,
so $f=\mathrm{id}$.
It follows that $L(\eta)=\mathrm{id}$.
Since $L$ is injective, $\eta=0$.
Thus $\Phi$ is injective.
We have therefore embedded $\mathbb Z^{(D_0)}\rtimes F$ into $F$.

\medskip
\Needspace{6\baselineskip}
\noindent\textbf{Passing from $D_0$ to $D$.}
Since $F$ fixes $0$ and $1$, we have
\[
\mathbb Z^{(D)}\rtimes F
\cong\mathbb Z^2\times\bigl(\mathbb Z^{(D_0)}\rtimes F\bigr).
\]
Conjugate the image of $\Phi$ by the affine map $a(t)=1/4+t/2$ and
extend each resulting map by the identity outside $[1/4,3/4]$.
This gives an embedding of $\mathbb Z^{(D_0)}\rtimes F$ into $F$
supported in that interval.
Place affine copies of the infinite cyclic group $\langle z\rangle$
in $[0,1/4]$ and $[3/4,1]$, again extending by the identity.
All three subgroups lie in $F$ and have disjoint supports, so together
they give the required embedding of
$\mathbb Z^2\times\bigl(\mathbb Z^{(D_0)}\rtimes F\bigr)$ into $F$.
\end{proof}
\FloatBarrier

\subsection*{Acknowledgments}
The author thanks the Mathematical Institute of the University of Oxford for its hospitality.

He is grateful to Sang-hyun Kim for sharing his draft \cite{KL} and for a helpful comment on this note.
This short note grew out of an attempt to understand that paper with the help of ChatGPT.
The figures were generated using ChatGPT.

\Needspace{9\baselineskip}

\noindent
\textsc{Koji Fujiwara}\\
Okinawa Institute of Science and Technology Graduate University,
Okinawa 904-0495, Japan\\
KUIAS/RIMS, Kyoto University, Kyoto, 606-8305, Japan\\
\texttt{koji.fujiwara@oist.jp; kfujiwara@math.kyoto-u.ac.jp}


\begin{thebibliography}{99}
\bibitem{BF}
J.~Belk and B.~Forrest,
\emph{Rearrangement groups of fractals},
Trans. Amer. Math. Soc. \textbf{372} (2019), no.~7, 4509--4552.




\bibitem{Chornyi}
M.~Chornyi,
Superharmonic functions on the Lamplighter graph of Thompson's
group $F$,
arXiv:1907.01440, 2019.


\bibitem{GH}
E.~Guentner and N.~Higson,
Weak amenability of CAT(0)-cubical groups,
Geom. Dedicata \textbf{148} (2010), 137--156.

\bibitem{GRT}
E. Guentner, E. Reckwerdt and R. Tessera,
\emph{Proper actions and weak amenability of classical
relatively hyperbolic groups},
Ann. Inst. Fourier \textbf{76} (2026), no.~4, 1369--1399.

\bibitem{JMMdS}
K.~Juschenko, N.~Matte Bon, N.~Monod, and M.~de la Salle,
\emph{Extensive amenability and an application to interval exchanges},
Ergodic Theory Dynam. Systems \textbf{38} (2018), no.~1, 195--219.


\bibitem{KL}
S.-H. Kim and Y. Lodha,
\emph{A note on the proof of nonamenability of $F$},
arXiv:2610.18360, 2026.


\bibitem{OpenAI}
OpenAI,
\emph{Thompson's group $F$ is nonamenable},
OpenAI Math Release preprint, 23 September 2026.

\bibitem{OzHyp}
N.~Ozawa,
Weak amenability of hyperbolic groups,
Groups Geom. Dyn. \textbf{2} (2008), no.~2, 271--280.

\bibitem{Oz}
N.~Ozawa,
\emph{Examples of groups which are not weakly amenable},
Kyoto J. Math. \textbf{52} (2012), no.~2, 333--344.

\end{thebibliography}
\end{document}